\documentclass{amsart}    %\documentclass[a4paper]{article}
\usepackage{amsmath,amsthm,amssymb}
\usepackage{enumerate}
\usepackage[all]{xy}
\usepackage{hyperref}

\usepackage[utf8]{inputenc}

\DeclareMathOperator{\End}{\mathrm{End}}
\newcommand{\R}{\mathbb{R}}
\newcommand{\Z}{\mathbb{Z}}
\newcommand{\aaac}{\mathfrak{r}_{2,c}}
\newcommand{\aaan}{\mathfrak{r}_{2,n}}
\newcommand{\aaa}{\mathfrak{r}_2}
\DeclareMathOperator{\ad}{\mathrm{ad}}
\DeclareMathOperator{\tr}{\mathrm{tr}}

\newcommand{\cov}{\nabla}

\newcommand{\ff}{|\theta|^2}

\newcommand{\fq}{|\theta|^4}
\newcommand{\fm}{|\theta|^{-4}}
\DeclareMathOperator{\grad}{\mathrm{grad}}
\newcommand{\hlie}{\mathfrak{h}}
\newcommand{\id}{\mathrm{Id}}

\newcommand{\liealg}{\mathfrak{g}}
\newcommand{\lie}{\mathcal{L}}
\DeclareMathOperator{\rank}{\mathrm{rank}}

\DeclareMathOperator{\trace}{\mathrm{tr}}

\newtheorem{theorem}{Theorem}[section]
\newtheorem{proposition}[theorem]{Proposition}
\newtheorem{proposition-definition}[theorem]{Proposition-Definition}

\newtheorem{corollary}[theorem]{Corollary}
\theoremstyle{remark}

\newtheorem{remark}[theorem]{Remark}

\theoremstyle{definition}
\newtheorem{definition}[theorem]{Definition}

\title{Integrable LCK manifolds}
\author[B. Cappelletti-Montano]{Beniamino Cappelletti-Montano}
 \address{Dipartimento di Matematica e Informatica, Universit\`a degli Studi di Cagliari, Via Ospedale 72, 09124 Cagliari, Italy}
 \email{b.cappellettimontano@unica.it}

\author[A. De Nicola]{Antonio De Nicola}
 \address{Dipartimento di Matematica, Universit\`a degli Studi di Salerno, Via Giovanni Paolo II 132, 84084 Fisciano, Italy}
 \email{antondenicola@gmail.com}

\author[I. Yudin]{Ivan Yudin}
 \address{University of Coimbra, CMUC, Department of Mathematics, 3001-501 Coimbra, Portugal}
 \email{yudin@mat.uc.pt}


\thanks{This work was partially supported by
Fondazione di Sardegna and Regione Autonoma della  Sardegna, Project STAGE and KASBA, by the Centre for Mathematics of the
University of Coimbra - UIDB/00324/2020, funded by the Portuguese Government through FCT/MCTES.
BCM and ADN are members of GNSAGA - Istituto Nazionale di Alta
Matematica}

\begin{document}

\begin{abstract}
We study a natural class of LCK manifolds that we call integrable LCK manifolds: those where the anti-Lee form $\eta$ corresponds to an integrable distribution.
 As an application  we obtain a characterization of  unimodular integrable LCK Lie algebras as  Kähler  Lie algebras equipped with suitable derivations.
   % $d\eta =  \eta \wedge \theta$.
\end{abstract}

\maketitle

%%%%%%%%%%%%%%%%%%%%%%%%%%%%%
\section{Introduction}
A locally conformal Kähler (LCK) manifold is a Hermitian manifold $(M, J, g)$ of dimension $2n+2$ such that the fundamental 2-form $\Omega$
and the Lee  $1$-form $\theta$ satisfy the identities
\[
d \Omega = \theta \wedge \Omega,\quad d\theta=0.
\]
LCK  manifolds are a natural generalization of Kähler manifolds. They have been studied by many authors since the foundational work of Vaisman in the ’70s (see for instance \cite{DrOr} and references therein).
The most studied subclass of LCK manifolds is the one of Vaisman manifolds, which are the locally conformal Kähler manifolds such that the Lee $1$-form is parallel. For every Vaisman manifold the anti-Lee $1$-form $\eta := -\theta \circ J$  gives rise to a contact structure on the kernel of the Lee $1$-form $\theta$, that is, to a maximally nonintegrable distribution. In this paper we are interested in the opposite case:  the one of a LCK structure where the anti-Lee form $\eta$ corresponds to an integrable distribution. In this case we will say that $(M, J, g)$ is an \emph{integrable LCK manifold}.

Our work takes inspiration from examples of Tricerri in \cite{tri82} of LCK structures on some Inoue surfaces. In \cite{belg00}  Belgun carried out a systematical analysis of locally conformal Kähler metrics on compact complex surfaces. His paper was groundbreaking since at the time it was conjectured that all non-Kähler compact complex surfaces admit LCK metrics.  Belgun showed that no Inoue surface admits a Vaisman metric and some of them even do not admit an LCK metric. Some of those Inoue surface types that admit a non-Vaisman LCK metric turn out to be integrable LCK manifolds according to our definition.
 The examples of Tricerri were further studied in  \cite{cordero, sawai, andrada}. 
 
Our main results are the following. In Theorem~\ref{deta} we show that an LCK manifold  $M^{2n+2}$ is integrable if and only if
\(
d \eta = f \eta \wedge \alpha,
\)
where the function $f$ is given by $\frac{\delta \theta}{\ff}  + n$. In particular, if the metric is Gauduchon ($\delta \theta=0$) and the Lee vector field has length one, we obtain that $f=n$. Moreover, in this case the commutator of the Lee and anti-Lee vector fields is given by
\begin{equation*}
\left[ U,V \right] =  n \,V.
\end{equation*}
We investigate the possibilities that Lee or anti-Lee vector field are Killing. We prove that if the Lee vector field $U$ is Killing then the manifold is not complete. If the anti-Lee vector field is Killing and has constant length, then the manifold cannot be compact.

Section~4 is devoted to the study of integrable LCK Lie algebras. We show that if $\liealg$ is an integrable LCK Lie algebra then it is a semidirect product
\[
\liealg \cong \left\langle U,V \right\rangle
\ltimes_\rho \hlie,
\]
 where $\hlie=\left\langle U,V \right\rangle^\perp$ is a Kähler ideal and $\rho$ is the adjoint representation.
We can say more if the Lie algebra $\liealg$ is unimodular. In this case we prove that $\liealg$ is solvable and $\hlie$ is abelian.  

 In Section~5 we identify integrable LCK Lie algebras among all semidirect products as above. We reduce the classification of all unimodular integrable LCK Lie algebras to the classification of pairs of even dimensional matrices satisfying suitable relations involving the complex structure (see \eqref{conditions}). We provide an example showing that not all  integrable LCK Lie algebras are unimodular.

  In Section~6 we consider the four-dimensional case. We classify 4-dimensional unimodular integrable LCK Lie algebras by showing that they consist of a 1-parameter family $\liealg_b$ of almost abelian Lie algebras 
and one isolated case $\mathfrak{d}_4$. We identify these LCK Lie algebras in the list of all 4-dimensional LCK solvable Lie algebras of~\cite{angella}. The 1-parameter family $\liealg_b$ was thoroughly studied in ~\cite{andrada} . In particular it was identified for which parameters $b$ the 1-connected
Lie group associated to $\liealg_b$ admits a cocompact discrete subgroup.
The integrable LCK Lie algebra  $\mathfrak{d}_4$ also corresponds to a compact integrable LCK manifold that had already been studied as a 4-dimensional LCK solvmanifold in \cite{cordero,sawai}.

In the final section, for every even dimension greater than $2$, we provide examples of integrable LCK manifolds which are not globally conformal Kähler.

%%%%%%%%%%%%%%%%%%%%%%%%%%%%%
%%%%%%%%%%%%%%%%%%%%%%%%%%%%%
\section{Preliminaries}
%%%%%%%Hermitian manifolds %%%%%%%%%Kähler
Recall that an \emph{almost Hermitian structure} on a manifold $M$ of even dimension $2(n+1)$ is a couple $(J, g)$, where $J$ is a $(1, 1)$ tensor field on $M$, $g$ is a Riemannian metric and
\begin{equation*}
J^2 = -\id, \; \; g(JX, JY) = g(X, Y),
\end{equation*}
for $X, Y \in {\mathfrak X}(M)$.
The fundamental $2$-form of $M$ is defined by
\[
\Omega (X, Y) = g(X, JY), \; \; \mbox{ for } X, Y \in {\mathfrak X}(M).
\]
A manifold $M$ endowed with an almost Hermitian structure is said to be an
\emph{almost Hermitian manifold}.
The almost Hermitian manifold $(M, J, g)$ is said to be:
\begin{itemize}
\renewcommand\labelitemi{--}
\item \emph{Hermitian} if the Nijenhuis torsion $N_J$ of $J$ vanishes; \item
\emph{Kähler} if $(M, J, g)$ is Hermitian and $d\Omega = 0$.
\end{itemize}
In the Kähler case the 2-form $\Omega$ is harmonic.
For a general Hermitian manifold
 one gives the following definition.
\begin{definition}
\label{leeoneform}
Given a Hermitian manifold $(M,J,g)$,
the $1$-form $\eta = \frac1{n}\delta \Omega $ is called the \emph{anti-Lee
$1$-form} and $\theta:= i_J\eta$ is called the \emph{Lee $1$-form}.
\end{definition}
Notice that $\eta = -i_J \theta$.
The anti-Lee $1$-form also has the following property that will be used
frequently in the paper.
%%%% iJ d eta=0 %%%%%%%%%%%%%%
\begin{proposition}
\label{ijdeta}
Let $(M,J,g)$ be a Hermitian manifold. Then $i_J d\eta =0$.
\end{proposition}
\begin{proof}
Let $x\in M$ and
 $W$ an open neighbourhood of $x$ such that $\theta$ is exact on
$W$, i.e. $\theta = df$ for some $f \in C^\infty(W)$. Then $i_J d \eta =- i_J d
i_J \theta = -i_j d i_j d f $. It is a direct computation to check that
$((i_J d)^2 f)(X,Y) = N_J(X,Y) f $. Since $J$ is integrable, we get that
$i_J d\eta =0$ at $x$.
\end{proof}
The \emph{Lee} and \emph{anti-Lee vector fields} $U,V$ are defined as the metric duals of $\theta, \eta$, respectively.
Clearly, one has $|V|^2=|U|^2=|\eta|^2=|\theta|^2$.

%%%% LCK manifolds %%%%%%%%%%%%%%%%%%
A Hermitian manifold $(M, J, g)$ is said to be
\emph{locally conformal Kähler} (LCK) if the fundamental 2-form $\Omega$
and the Lee  $1$-form $\theta$ satisfy the identities
\[
d \Omega = \theta \wedge \Omega,\quad d\theta=0.
\]
For every LCK manifold one can show that
\begin{equation*}
\begin{aligned}
i_U \Omega = -\eta,\quad i_V \Omega = \theta.
\end{aligned}
\end{equation*}

%%%% Vaisman manifolds %%%%%%%%%%%%%%%%%%
An LCK manifold is said to be a \emph{Vaisman manifold}  if  the Lee 1-form $\theta$ is parallel with respect to the Levi-Civita connection of $g$.
On every Vaisman manifold the Lee vector field $U$ is Killing and is an infinitesimal
automorphism of the complex structure. Hence
\(
{\mathcal L}_U \Omega = 0
\)
(see, for instance, Propositions 4.2 and 4.3 in \cite{DrOr}; see also \cite{vaisman_roma}).
Now, recall the following definition from \cite{lcs-Vai}.
\begin{definition}%[Vaisman]
A \emph{locally conformal symplectic (LCS) structure of the first kind} on a manifold $M$ of dimension $2n+2$  is given by a triple $(\Omega,\theta,U)$
where $\Omega$ is a nondegenerate 2-form  such that
 \(
d\Omega = \theta \wedge \Omega,
\)
for some closed 1-form $\theta$, and $U$ is a vector field such that $\theta(U)\neq 0$ and
 \(
 \lie_U\Omega=0
 \).
\end{definition}

Clearly, from the above definitions one can see that a Vaisman manifold has an underlying LCS structure of the first kind. An alternative equivalent characterization of an LCS structure of the first kind is the following one.
\begin{definition}
A \emph{LCS structure of the first kind} on a manifold $M$ of dimension $2n+2$ is a pair $(\theta, \eta)$ of $1$-forms such that:
\begin{enumerate}[$(i)$]
\item
$\theta$ is closed;
\item
the rank of $d\eta$ is $2n$ and $\theta \wedge \eta \wedge (d\eta)^{n}$ is a volume form.
\end{enumerate}
\end{definition}

So, in any LCK structure with an underlying LCS structure of the first kind
(e.g. in any Vaisman manifold) the anti-Lee $1$-form $\eta$  induces a contact
structure on the kernel of the Lee $1$-form $\theta$. In other words, the
distribution $\ker \eta$ is  maximally nonintegrable.
In this paper we will consider the opposite case, where $\ker\eta$ is integrable.

\strut

%{\bf *** LCK WITH KILLING ANTI-LEE***}
In the remaining part of this section we prove two properties for general LCK
manifolds that are useful to study the case when the anti-Lee vector field is Killing.
\begin{proposition}
\label{lieV}
Let $(M,J,g)$ be an LCK manifold.
Then
\begin{equation}
\label{lievg}
(\lie_V g)^\# = (\lie_V J)\circ J.
\end{equation}
\end{proposition}

\begin{proof}
Recall that in any LCK manifold one has $\lie_V \Omega=0$. Thus
\begin{equation*}
\begin{aligned}
0 = \lie_V (g\circ (\id\otimes J)) = (\lie_V g) (\id \otimes J) + g \circ
(\id \otimes \lie_V J).
\end{aligned}
\end{equation*}
Multiplying the above equality by $\id \otimes J$ on the right hand we get
\begin{equation*}
\begin{aligned}
0 = - (\lie_V g) (\id\otimes \id) + g \circ (\id \otimes (\lie_V J)\circ J).
\end{aligned}
\end{equation*}
Thus
\begin{equation*}
\lie_V g = g \circ (\id \otimes (\lie_V J)\circ J).
\end{equation*}
By raising the index we get the claim.
\end{proof}
As a consequence, when the anti-Lee vector field $V$ is Killing we obtain that it commutes with the Lee vector field $U$.
\begin{corollary}
\label{vkilling}
Let $(M,J,g)$ be an LCK manifold such that the anti-Lee vector field $V$ is Killing.
Then $\lie_V J=0$ and hence
\begin{equation*}
[U, V]=0.
\end{equation*}
\end{corollary}
\begin{proof}
Let $V$ be Killing. Then by Proposition ~\ref{lieV} we have
\begin{equation*}
\lie_V J=-(\lie_V g)^\#\circ J = 0.
\end{equation*}
Thus
\begin{equation*}
0=(\lie_V J)U=[V, V] - J [V, U]=- J [V, U].
\end{equation*}
Hence $U$ and $V$ commute.
\end{proof}

%%%%%%%%%%%%%%%%%%%%%%%%%%%%%
%%%%%%%%%%%%%%%%%%%%%%%%%%%%%
\section{Integrable LCK manifolds}
So far, the most studied class of LCK manifolds are Vaisman manifolds, for
which the distribution $\ker \eta$ is maximally non-integrable.
In this paper we will study the class of LCK manifolds for which $\ker \eta = V^\perp$ is
integrable.

Notice, that to be able to speak about the integrability of $V^\perp$, we have to
assume that $V$ is non-zero everywhere, as otherwise $\ker\eta$
cannot be a smooth distribution, as its dimension jumps in the points where $V_x =0$.

\begin{definition}
We say that an LCK manifold $(M, J,g)$ is \emph{integrable} if the anti-Lee vector field $V$ is everywhere
non-zero and $V^\perp$ is an integrable distribution.
\end{definition}

\begin{proposition}
\label{sss}
Let $(M,J,g)$ be an integrable
 LCK manifold.
Then $|\theta|^4 d\eta = d\eta(U,V)\cdot \theta \wedge \eta$.
\end{proposition}
\begin{proof}
Since $V^\perp = \ker \eta$, we get that the integrability of $V^\perp$ is
equivalent to $\eta \wedge d\eta =0$.
Applying $i_V$ to $\eta \wedge d\eta =0$, we get
$\ff d\eta - \eta \wedge i_Vd\eta =0$, i.e. $\ff d\eta = \eta \wedge i_Vd\eta$.

Next apply $i_J$ to $\eta\wedge d\eta=0$. Since $i_J \eta =\theta$
and by Proposition~\ref{ijdeta} one has $i_Jd\eta=0$, we get $\theta \wedge d\eta=0$.
Applying $i_{{}_U}$ to the last equation we obtain $\ff d\eta - \theta \wedge i_{{}_U}d\eta$,
i.e. $\ff d\eta = \theta \wedge i_{{}_U} d\eta$.

Combining it with $\ff d \eta = \eta \wedge i_V d\eta$, we get $\eta \wedge
i_Vd\eta = \theta \wedge i_{ {}_U}d\eta$. Applying $i_V$, we
obtain
$\ff \cdot i_Vd\eta=\theta \cdot d\eta(V,U)$. Reusing $\ff d\eta = \eta \wedge
i_Vd\eta$, we get
\begin{equation*}
|\theta|^4 d\eta =  \eta \wedge( \ff \cdot i_Vd\eta) = \eta \wedge
d\eta(V,U) \cdot \theta = d\eta(U,V) \cdot \theta \wedge \eta.
\end{equation*}
\end{proof}
%%%%%%%%%%%%%%%%%%%%
The next result will allow to compute the term $d\eta(U,V)$ which appears in the above expression of $d\eta$
and will be used to compute the commutator of the Lee and the anti-Lee vector fields in any integrable LCK manifold.
%%%%%%%%%%%%%%%%%%%%%%%% new version from correction
\begin{theorem}\label{delta}
Suppose $M^{2n+2}$ is an integrable LCK manifold.
Then
\begin{equation}\label{eq:delta} %true delta?
\ff \delta \theta + d\eta(U,V) + n |\theta|^4  =0.
\end{equation}
In particular, if $M^{2n+2}$  has unitary Lee  vector field one gets
\begin{equation}\label{eq:delta1}
\delta \theta + d\eta (U,V) + n =0.
\end{equation}
\end{theorem}
\begin{proof}
Let $x\in M$. Choose a neighbourhood of $x$ where
$\theta$ does not vanish and there is an orthonormal frame
$X_1$, \dots $X_{2n+2}$,  such that $X_1 = |\theta|^{-1}U$ and $JX_k =
X_{n+1+k}$ for $1\le k \le n+1$. In particular,  $X_{n+2} = |\theta|^{-1}V$.
It is well known that codifferential $\delta$ can be expressed as $-\sum_{k=1}^{2n+2}
i_{X_k} \cov_{X_k}$. Applying this expression to $\theta = g\circ (U\otimes \id)$, we get
\begin{equation*}
\begin{aligned}
\delta \theta & = -\sum_{k=1}^{2n+2} g (\cov_{X_k} U, X_k).
\end{aligned}
\end{equation*}
Notice that we also have
\begin{equation*}
\begin{aligned}
\sum_{k=1}^{2n+2} g([U,X_k], X_k) = - \sum_{k=1}^{2n+2} g(\cov_{X_k} U , X_k),
\end{aligned}
\end{equation*}
since $2g(\cov_U X_k, X_k) =-(\cov_U g)(X_k,X_k) + U(g(X_k,X_k)) =0 $ for all
$k$.
Thus
\begin{equation}
\label{deltatheta}
\begin{aligned}
\delta \theta & = \sum_{k=1}^{2n+2} g([U,X_k], X_k).
\end{aligned}
\end{equation}
As the next step we relate $g([U,X_k],X_k) + g([U,JX_k],JX_k)$ with
$d\Omega(X_k,JX_k,U)$ for $2\le k \le n+1$. We get
\begin{equation*}
\begin{aligned}
d\Omega(X_k,JX_k,U) &= (\theta \wedge \Omega) (X_k,JX_k,U)=
 \Omega(X_k,JX_k)\theta(U) \\ & = g(X_k,J^2X_k )\ff = -
\ff.
\end{aligned}
\end{equation*}
 Now we compute $d\Omega(X_k,JX_k,U)$ using the definition of exterior derivative
\begin{align*}
d\Omega(X_k,JX_k,U) & = X_k (\Omega(JX_k,U)) - (JX_k)(\Omega(X_k,U)) +
U(\Omega(X_k,JX_k))\\[2ex] &
\phantom{=}  -  \Omega([X_k,JX_k], U) -  \Omega([JX_k,U], X_k) -  \Omega([U,X_k], JX_k).
\end{align*}
The first three terms in the above sum vanish, and we get
\begin{equation*}
d\Omega(X_k,JX_k,U) = -  i_U\Omega([X_k,JX_k]) +  g([U,JX_k], JX_k)
+  g([U,X_k], X_k).
\end{equation*}
By separate computation we have
\begin{equation*}
 i_U\Omega([X_k,JX_k]) = - \eta ([X_k,JX_k]) = 0
\end{equation*}
since the integrability of $V^\perp$
implies that $[X_k,JX_k]$ is orthogonal to $V$, as both  $X_k$ and $JX_k$ are orthogonal to $V$.
Thus
for $2\le k \le n+1$
\begin{equation}
\begin{aligned}
g([U,JX_k], JX_k) +  g([U,X_k], X_k) = -\ff.
\end{aligned}
\label{uxk}
\end{equation}

Now we compute
$g([U,X_{n+2}], X_{n+2}) +  g([U,X_1], X_1)$.
Since $X_1$ is proportional to $U$, we have that $d\Omega(X_1, X_{n+2},U)=0$. Hence
\begin{equation*}
\begin{aligned}
0 &= d\Omega\left( X_1,X_{n+2},U \right)  =
X_1 (\Omega(X_{n+2},U)) - X_{n+2}(\Omega(X_1,U)) + U(\Omega(X_1,X_{n+2}))\\[1ex] & \phantom{====} -  \Omega([X_1,X_{n+2}], U) -  \Omega([X_{n+2},U], X_1) -  \Omega([U,X_1], X_{n+2})
\\[2ex] & = X_1 (g (X_{n+2},V)) - X_{n+2}(g(X_1,V)) + U(0) \\[1ex] &
\phantom{====}  + i_U\Omega ([X_1,X_{n+2}]) + g([U,X_{n+2}],X_{n+2}) + g([ U,X_1
],X_1)
\\[2ex] & = X_1 (\eta(X_{n+2})) - X_{n+2}(\eta(X_1))   - \eta ([X_1,X_{n+2}]) + g([U,X_{n+2}],X_{n+2}) + g(\left[ U,X_1 \right],X_1)
\\[2ex] & = d\eta (X_1,X_{n+2}) + g([U,X_{n+2}],X_{n+2}) + g(\left[ U,X_1 \right],X_1)
\end{aligned}
\end{equation*}
Hence
\begin{equation}
\label{ux1}
\begin{aligned}
g([U,X_{n+2}],X_{n+2}) + g(\left[ U,X_1 \right],X_1) & =-   d\eta (X_1,X_{n+2})
 = - |\theta|^{-2} d\eta(U,V).
\end{aligned}
\end{equation}
Substituting \eqref{uxk} and \eqref{ux1} in \eqref{deltatheta},
we get
\begin{equation*}
\delta \theta = -n \ff - |\theta|^{-2} d\eta(U,V).
\end{equation*}
Multiplying with $\ff$, this gives
\begin{equation*}
\ff\delta \theta + d\eta(U,V) + n |\theta|^4  =0.
\end{equation*}
\end{proof}
Now we are ready to give a better expression for $d\eta$ and to compute the commutator $[U, V]$.
\begin{theorem}\label{deta}
Let $(M^{2n+2}, J, g)$ be an integrable LCK manifold.
Then
\begin{align*}
d\eta & = \left(\frac{\delta \theta}{\ff}  + n \right) \eta \wedge \theta,\\
\left[ U,V \right] & = \big(\,\delta \theta  + n \ff\,\big)V + J( \grad \ff).
\end{align*}
\end{theorem}
\begin{proof}
The first equation is an immediate consequence of Proposition~\ref{sss} and
Theorem~\ref{delta}.
To prove the second formula we compute the components of left and right sides
of the equation with respect
to the orthogonal decomposition $\left\langle U \right\rangle\oplus \left\langle
V \right\rangle \oplus \left\langle U,V \right\rangle^\perp$ of the space of
vector fields.  For the $V$-component,  by using Theorem~\ref{delta}, we get
\begin{equation*}
\begin{aligned}
g(\left[ U,V \right], V)  & =  \eta([U,V]) = U(\eta(V)) - d\eta(U,V)
= U(\ff) - \ff( \delta \theta + n \ff)
\\[2ex]\ & = -d(\ff) (JV)  - \ff( \delta \theta + n \ff)
\\[2ex]\ & = -g ( \grad\ff, JV)  - \ff( \delta \theta + n \ff)
\\[2ex]\ & = g \big( J\grad\ff, V\big)  - g(V,V) (\, \delta \theta + n\, \ff)
\\[2ex]\ & = g \big(\, J\grad\ff - ( \delta \theta + n \ff)V\,,\,\, V\,\big).
\end{aligned}
\end{equation*}
Next, applying $g(-,U)$ we obtain
\begin{equation*}
\begin{aligned}
g(\left[ U,V \right], U)  & =  \theta([U,V]) =- V(\theta(U)) - d\theta(U,V)
= - V(\ff)
\\[2ex] & = - d(\ff)(JU) = - g\big(\, \grad\ff , JU) =
\\[2ex] & = g \big(\, J\grad\ff - ( \delta \theta + n \ff)V\,,\,\, U\,\big).
\end{aligned}
\end{equation*}
Finally, for $X \in \left\langle U,V \right\rangle^\perp$
\begin{equation}
\label{guvx}
\begin{aligned}
g(\left[ U,V \right], X)  & =  - \Omega ([U,V],JX)
\\[2ex]  &  = d\Omega(U,V,JX) - U
(\Omega (V,JX)) + V (\Omega(U,JX))
\\[1ex] & \phantom{=}  - (JX) (\Omega (U,V)) + \Omega ( [V,JX], U)
+ \Omega ( [JX,U],V).
\end{aligned}
\end{equation}
Now, we calculate each term of the above formula
\begin{equation}\label{omegauvjx}
\begin{aligned}
d\Omega(U,V,JX) = (\theta \wedge \Omega)(U,V,JX) = \ff \Omega(V,JX) = - \ff
g(V,X) =0.
\end{aligned}
\end{equation}
\begin{equation}\label{jomegavjx}
\begin{aligned}
U(\Omega(V,JX)) =- U(g(V,X)) = 0,\quad V(\Omega(U,JX)) = -V(g(U,X)) =0.
\end{aligned}
\end{equation}
\begin{equation}
\label{jxomegauv}
\begin{aligned}
(JX)(\Omega(U,V)) = - (JX)(\ff) = - g\big( \grad\ff, JX) = g( J\grad \ff, X).
\end{aligned}
\end{equation}
\begin{equation}
\label{omegavjxu}
\begin{aligned}
\Omega ( [V,JX], U) & = g( [ V,JX ], V)
 = \eta ([ V,JX ])\\[2ex] &  =
- d\eta (V,JX) + V ( \eta (JX)) - (JX) (\eta (V))
\\[2ex] &= - (\delta \theta + n \ff) (\eta \wedge \theta ) (V,JX) + 0 - (d\ff)(JX)
\\[2ex] &= 0 - g ( \grad \ff, JX) \\[2ex] &  = g ( J\grad \ff, X).
\end{aligned}
\end{equation}
\begin{equation}\label{omegajxuv}
\begin{aligned}
\Omega ( [JX,U],V) & = - g( [JX,U], U) = - \theta ( [JX,U])
\\[2ex] & =  d\theta (JX,U) - (JX) (\ff) + U (\theta (JX))\\[2ex] &  = -(d\ff) (JX)
= - g ( \grad \ff, JX) \\[2ex] &  =  g ( J \grad \ff, X).
\end{aligned}
\end{equation}
Substituting (\ref{omegauvjx}$-$\ref{omegajxuv}) in \eqref{guvx}, we
get
\begin{equation}
\begin{aligned}
g(\left[ U,V \right], X)  & = g( J \grad \ff , X)  = g \big(\, J\grad\ff - ( \delta \theta + n \ff)V\,,\,\, X\,\big). \end{aligned}
\end{equation}
This completes the proof of the theorem.
\end{proof}

In any integrable LCK manifold the distribution $\left\langle U,V \right\rangle^\perp$ is always integrable, as $\left\langle U,V \right\rangle^\perp= \ker\eta\cap\ker\theta$ and $d \theta=0$. Recall that an LCK manifold is a \emph{Gauduchon} Hermitian manifold if and only
if $\delta \theta =0$. In this case if the norm of $\theta$  is unitary the
above theorem implies that also the complementary distribution $\left\langle U,V \right\rangle$ is integrable, as it is shown in the next result.
\begin{corollary}
Let $(M^{2n+2},J,g)$ be an integrable LCK manifold such that the metric is Gauduchon  and $|\theta| =1$.
Then
\begin{equation*}
\begin{aligned}
d\eta  =  n \, \eta \wedge \theta\\
\left[ U,V \right]  =  n \,V.
\end{aligned}
\end{equation*}
\end{corollary}
%\begin{proof}
%The two above equations are immediate consequences of Theorem~\ref{deta}. We will now show that $J( \grad \ff) \in \left\langle U,V \right\rangle$.
%\end{proof}

%{ \bf Can we prove something interesting for integrable LCK with $\theta$
%harmonic and $|\theta|^2 =1$. ($\theta$ harmonic means that metric is
%Gauduchon.)   }
%
% \strut

One interesting special case of LCK manifolds is obtained when the Lee vector field is Killing. However, this cannot happen
 in a compact or, more generally, complete integrable LCK manifold.
\begin{theorem}
\label{ukillingcomplete}
Let $(M,J,g)$ be an integrable
 LCK manifold such that the Lee vector field $U$ is Killing.
Then $M$ is not complete.
\end{theorem}
\begin{proof}
By Theorem~\ref{deta} we have
\begin{equation}\label{nablaUV-nablaVU}
\nabla_U V- \nabla_V U = \big(\,\delta \theta  + n \ff\,\big)V + J( \grad \ff).
\end{equation}
Since $U$ is Killing we have that $\nabla U$ is $g$-skew symmetric
\begin{equation*}
g(\nabla_V U, V)=g(\nabla U (V), V)=0.
\end{equation*}
Moreover, as $\theta$ is the metric dual of $U$ one has $\delta \theta=0$. Hence, if we take the scalar product by $V$ of equation \eqref{nablaUV-nablaVU} we get
\begin{align*}
 g(\nabla_U V, V) & =    n \theta^4  + g(J( \grad \ff),V)\\
 			    &=   n \theta^4  + g(J( \grad \ff),JU)\\
			    &= n \theta^4  + g(\grad \ff,U)
\end{align*}
that is
\begin{equation}\label{nablaUV}
g(\nabla_U V, V) = n \fq  + U(\ff).
\end{equation}
Now,
\begin{equation*}
U(\ff)= U(g(V,V))=2 g(\nabla_U V, V).
\end{equation*}
Thus
\begin{equation*}
g(\nabla_U V, V) =\frac12 U(\ff).
\end{equation*}
So, equation~\eqref{nablaUV} becomes
\begin{equation*}
U(\ff) = -2 n \fq
\end{equation*}
that implies
\begin{equation}\label{uff}
\fm U(\ff)  = -2 n.
\end{equation}
Now, let $p\in M$ and $\gamma:I\to M$ be the maximal integral curve of $U$ such that $\gamma(0)=p$. Define
\begin{equation*}
h(t)=\frac{1}{|\theta|^2_{\gamma(t)}}=\frac{1}{|U|^2_{\gamma(t)}}=\frac{1}{|\dot\gamma(t))|^2}.
\end{equation*}
Hence
\begin{equation*}
\frac{d h}{dt}(\gamma(t)) =- |\theta^{-4}| d(|\theta|^2)(\dot\gamma(t))
						=-( |\theta^{-4}| U(|\theta|^2))(\gamma(t)).
\end{equation*}
Thus ~\eqref{uff} implies
\begin{equation*}
\frac{dh}{dt}    (t)= 2n
\end{equation*}
Thus  we get
\begin{equation*}
\frac{1}{|\dot\gamma(t))|^2}= 2 n t + c
\end{equation*}
where $c$ is a real constant. Hence
\begin{equation*}
|\dot\gamma(t))|^2=\frac{1}{2 n t + c}
\end{equation*}
so that  $|\dot\gamma(t))|^2$ is not defined for $t=-\frac{c}{2n}$. Hence $U$ does not admit global integral curves. Therefore $M$ cannot be a complete manifold.
\end{proof}

Next we  consider the  case of an  integrable LCK manifold such that the anti-Lee vector field $V$ is Killing. In this case we can consider the transversal structure with respect to the foliation generated by $V$.
In this case,  by Corollary~\ref{vkilling} we get $[U,V]=0$. Therefore, from Theorem~\ref{deta} we have
\begin{equation*}
\big(\,\delta \theta  + n \ff\,\big)V =- J( \grad \ff).
\end{equation*}
We obtain that if $M$ is compact then $\grad \ff$ must be nonzero. Otherwise we would get
\begin{equation}
\delta \theta  =- n \ff\
\end{equation}
that gives a contradiction, as it implies $\Delta \theta=0$ and hence $\delta \theta=0$. We obtain the following result.
\begin{corollary}
Let $(M^{2n+2},J,g)$ be an integrable LCK manifold such that the anti-Lee vector field $V$ is Killing and  $|\theta|$ is constant.
Then  $M$ is not compact.
\end{corollary}

\section{Integrable  LCK Lie algebras.}

An LCK structure on a Lie algebra $\liealg$ is given by $\Omega \in \bigwedge^2
\liealg^*$, $g \in S^2 \liealg^*$,  $\theta, \eta\in \liealg^*$, $J \in \End(\liealg)$ such that
\begin{enumerate}[i)]
\item $g$ is positively defined;
\item $\Omega = g \circ (\id \otimes J)$;
\item $d\Omega = \Omega \wedge \theta$, where $d$ is Chevalley-Eilenberg
differential;
\item $d \theta=0$;
\item $\eta = -i_J \theta$;
\item $J$ is a complex structure, i.e. $J^2 = -\id$ and  $N_J(X,Y) =0$ for all $X$, $Y\in \liealg$.
\end{enumerate}
We will denote by $U$ the metric dual of $\theta$ and by $V$ the metric dual of
$\eta$. To have an LCK structure on $\liealg$ is the same as to have a right
invariant LCK structure on the connected and  simply connected  Lie group $G$
with the Lie algebra $\liealg$.
The LCK Lie algebra $\liealg$ is said to be \emph{integrable} if $\eta \wedge d\eta=0$.

Notice that given a right-invariant LCK structure on a Lie group $G$, we have
$|\theta|$ constant. By rescaling the metric we can always assume that
$|\theta|=1$.

\begin{proposition}
\label{kaehlerideal}
Let $\liealg$ be an integrable LCK Lie algebra. Then $\hlie := \left\langle U,V
\right\rangle^\perp$ is a Kähler ideal in $\liealg$.
\end{proposition}
\begin{proof}
First we show that $\hlie$ is an ideal in $\liealg$.
Let $X\in\liealg$ and $Y\in \hlie$. Then
\begin{equation*}
\begin{aligned}
g([X,Y],U) & = \theta([X,Y])  = - d\theta (X,Y) = 0, \\
g([X,Y], V) & = \eta([X,Y])  = - d\eta (X,Y).
\end{aligned}
\end{equation*}
From Theorem~\ref{deta} we have that $d\eta(X,Y)$ is proportional to $(\eta
\wedge \theta)(X,Y)$, which is zero
since $\eta(Y) = g(Y,V) =0$ and  $\theta(Y)=g(Y,U) =0$ for $Y \perp \left\langle
U,V \right\rangle$. Hence $[X,Y]$ is orthogonal to both $U$ and $V$, i.e.
$\left[ X,Y \right] \in \hlie$.
Thus $\hlie$ is an ideal.

We have $J\hlie \subset \hlie$, since for every $Y \in \hlie$
\begin{equation*}
\begin{aligned}
g(JY,U) & = - g(Y,JU) = - g(Y,V) =0\\
g(JY,V)&  = - g(Y,JV) =  g(Y,U) =0.
\end{aligned}
\end{equation*}
Hence it induces an almost complex structure $J'$ on $\hlie$. Moreover, since $J$  is integrable then  $J':\hlie\to\hlie$ is also integrable. Next, define $\omega$ as the pullback of $\Omega$ via the inclusion $i:\hlie\to\liealg$
to $\hlie$. For any triple $X$, $Y$, $Z\in \hlie$, we get
\begin{equation*}
d\omega (X,Y,Z)=d i^* \Omega(X,Y,Z)= i^* d\Omega(X,Y,Z) = i^* (\theta \wedge \Omega)(X,Y,Z) = 0.
\end{equation*}
It is left to check that $g'(X,Y) = \omega(X,J'Y)$ is positive definite and Hermitian.
But $g'$ is just the restriction of $g$ to $\hlie$, so it is positive definite. Moreover,
\begin{equation*}
\begin{aligned}g'(Y,J'Z) & = g(Y,JZ) = - g(JY, Z) = - g'(J'Y,Z),
\end{aligned}
\end{equation*}
for $Y$, $Z\in \hlie$.
This completes the proof of the fact that $(g',J',\omega)$ is a Kähler structure on
$\hlie$.
\end{proof}
Let $\liealg_1$ and $\liealg_2$ be Lie algebras. An action of $\liealg_1$ on
$\liealg_2$ is a linear map
$\rho\colon \liealg_1 \otimes \liealg_2 \to \liealg_2$, such that the
corresponding map $\tilde\rho \colon \liealg_1 \to \End(\liealg_2)$ is a
homomorphism of Lie algebras and its image lies in the subspace of
derivations on $\liealg_2$.
Given an action $\rho$ of $\liealg_1$ on $\liealg_2$ we define a bracket on
$\liealg_1 \oplus \liealg_2$ by
\begin{equation*}
\left[ (a_1,b_1), (a_2,b_2) \right] = \left( \left[ a_1,a_2 \right] ,
\rho(a_1\otimes b_2) - \rho(a_2\otimes b_1)+ \left[ b_1,b_2 \right] \right).
\end{equation*}
It is well known that this bracket gives a structure of Lie algebra on
$\liealg_1 \oplus \liealg_2$. The resulting Lie algebra is denoted by $\liealg_1
\ltimes_\rho \liealg_2$ and is called a \emph{semi-direct} product of
$\liealg_1$ and $\liealg_2$.

\begin{corollary}
Let $\liealg$ be an integrable LCK Lie algebra and $\hlie=\left\langle U,V
\right\rangle^\perp$. Then $\liealg \cong \left\langle U,V \right\rangle
\ltimes_\rho \hlie$, where $\rho(U) = ad_U|_\hlie$ and $\rho(V) = ad_V |_\hlie$.
\label{semidirect}
\end{corollary}
\begin{proof}
Due to Proposition~\ref{kaehlerideal},
the only thing to check is that $\left\langle U,V \right\rangle$ is a Lie
subalgebra of $\liealg$. By Theorem~\ref{deta}, we have for the corresponding
right-invariant LCK structure on the $1$-connected Lie group of $\liealg$
\begin{equation*}
\begin{aligned}
\left[ U,V \right] = (\delta\theta + n \ff)V + J(\grad|\theta|^2) = (\delta \theta
+ n) V,
\end{aligned}
\end{equation*}
since $|\theta|=1$ by our assumption.
Hence $\left\langle U,V \right\rangle$ is a Lie subalgebra.
\end{proof}

We can say more about an integrable LCK Lie algebra if it is unimodular, that is, if $\tr(ad_X)=0$, for each $X\in\liealg$.
In~\cite{hano}, Hano proved that every unimodular Kähler Lie algebra $\hlie$ is
meta-abelian, i.e. $\hlie^{(2)} =0$. In particular, $\hlie$ is solvable. We
will show that there is a similar result for integrable LCK Lie algebras.
\begin{theorem}\label{unimodularsolvable}
Let $\liealg$ be a unimodular integrable LCK Lie algebra. Then $\liealg$ is
solvable.
\end{theorem}
\begin{proof}
By Proposition~\ref{kaehlerideal} $\hlie = \left\langle U,V
\right\rangle^\perp$ is a Kähler ideal in $\liealg$. Hence by the above
mentioned Hano's result $\hlie^{(2)} =0$.

Next we show that $\liealg^{(2)} \subset \hlie$. This will immediately imply
that $\liealg^{(4)}=0$, i.e. that $\liealg$ is solvable.
We have $\liealg^{(1)} \subset U^\perp$. Indeed, for any $X$, $Y\in \liealg$
\begin{equation*}
\begin{aligned}
g(\left[ X,Y \right], U) = \theta([X,Y]) = - d\theta(X,Y) =0.
\end{aligned}
\end{equation*}
Hence $\liealg^{(2)} \subset \left[ U^\perp, U^\perp \right]$ and
\begin{equation*}
\begin{aligned}
\left[ U^\perp , U^\perp \right] & = \left[ \left\langle V \right\rangle\oplus
\hlie, \left\langle V \right\rangle\oplus \hlie \right]
\subset \left\langle \left[ V,V \right] \right\rangle + \left[ V, \hlie \right]
+ \left[ \hlie, \hlie \right] \subset \hlie + \hlie^{(1)} = \hlie
\end{aligned}
\end{equation*}
imply $\liealg^{(2)} \subset \hlie$.
This finishes the proof.
\end{proof}
For an integrable LCK Lie algebra $\liealg^{2n+2}$, by Theorem~\ref{deta} we
get $\left[ U,V \right]= cV$, where
the constant
$c$ is given by $\delta \theta + n$.
In the case $\liealg$ is unimodular, we have
\begin{equation}
\label{uvnv}
\begin{aligned}
\left[ U,V \right] = nV.
\end{aligned}
\end{equation}
Indeed, arguing like in the proof of Theorem~\ref{delta} we get that in every
Riemannian Lie algebra holds an analogue of~\eqref{deltatheta}, namely  $\delta \theta =
\trace(\ad_U)$, which is zero since $\liealg$ is unimodular.

In the next theorem we give an explicit description of $\liealg^{(1)}$ for
unimodular integrable LCK Lie algebras.
We will use it later to prove that $\hlie$ is an abelian ideal in every
unimodular integrable LCK Lie algebra.
\begin{theorem}\label{uperp}
Let $\liealg$ be an integrable LCK Lie algebra, $U$ its Lee
vector, and $V = JU$. If $[U,V] = cV $ with $c>-1$, $c\not=0$, then $\left[ \liealg, \liealg \right] = U^\perp$.
In particular,  $\left[ \liealg, \liealg \right] = U^\perp$
for every unimodular integrable LCK Lie
algebra.
\end{theorem}
\begin{proof}
We have $\left[ \liealg,\liealg \right] \subset U^\perp$, since for any
$X$, $Y\in \liealg$
\begin{equation*}
\begin{aligned}
g([X,Y], U) = \theta([X,Y]) = - d\theta (X,Y) = 0.
\end{aligned}
\end{equation*}
Thus we only have to check the inclusion $U^\perp \subset \left[ \liealg, \liealg \right]$.
Observe that $V = [(1/c)U,V] \in \left[ \liealg,\liealg \right]$.
We get  the vector space filtration of  $\liealg$
\begin{align}
\left\langle V \right\rangle \subset \left[ \liealg,\liealg \right] \subset
U^\perp\subset \liealg.
\end{align}
If the inclusion $\left[ \liealg,\liealg \right] \subset U^\perp$ is
proper,
then there exists a non-zero $X \in  \left[ \liealg, \liealg \right]^\perp \cap U^\perp$.
We are going to show that no such $X$ can exist if $c>-1$ and $c\not=0$.

Suppose $X \in \left[ \liealg,\liealg \right]^\perp \cap U^\perp \subset
V^\perp \cap U^\perp$, $X\not=0$.
Without loss of generality we can assume $g(X,X) =1$.

First we show $\left[ U,JX \right] =-JX$. We  start by showing  that $\left[ U, JX
\right]$ is in $\left\langle U,V \right\rangle^\perp$. We will use that
$JX \in \left\langle U,V \right\rangle^\perp$, which holds since $X \in
\left\langle U,V \right\rangle^\perp$ and $\left\langle U,V
\right\rangle^\perp$ is $J$-invariant. We have
\begin{equation*}
\begin{aligned}
g([U,JX],U) & = \theta([U,JX]) = -d\theta (U,JX) =0\\
g([U,JX],V) & = \eta([U,JX]) = - d\eta (U,JX) = - c (\eta \wedge
\theta)(U,JX) =0.
\end{aligned}
\end{equation*}
We also have $[U,JX] \perp X$, since $X \in \left[ \liealg,\liealg
\right]^\perp$. Next we check that $g([U,JX], JX) =-1$. We have
\begin{equation}
\label{gujxjx}
\begin{aligned}
g([U,JX],JX) & = \Omega ([U,JX], X) \\
&= - d\Omega(U,JX,X) - \Omega([JX,X],U) - \Omega([X,U],JX) \\
& = - (\theta\wedge \Omega)(U,JX,X) - g([JX,X],V) + g([X,U],X).
\end{aligned}
\end{equation}
The last term in the above expression vanishes since $X \in \left[ \liealg,
\liealg
\right]^\perp$. Further
\begin{equation*}
\begin{aligned}
g([JX,X],V) = \eta([JX,X]) = - d\eta(JX,X) = -c (\eta \wedge \theta)(JX,X) =0.
\end{aligned}
\end{equation*}
 Thus, returning to the computation started
in~\eqref{gujxjx}, we get
\begin{equation*}
\begin{aligned}
g([U,JX],JX) & = - (\theta \wedge \Omega)(U,JX,X) = - \Omega(JX,X) = -
g(JX,JX) =-1.
\end{aligned}
\end{equation*}
It is left to check that $[U,JX] \in \left\langle U,V, X,JX
\right\rangle$. Let $Y \in \left\langle U,V,X,JX \right\rangle^\perp$.
Then
\begin{equation*}
\begin{aligned}
g([U,JX],Y) & = - \Omega([U,JX],JY) \\& =  d\Omega(U,JX,JY) + \Omega([JX,JY],U) +
\Omega([JY,U],JX)
\\& =  (\theta\wedge \Omega)(U,JX,JY) + g([JX,JY],V) - g([JY,U],X).
\end{aligned}
\end{equation*}
The last term vanishes once again since $X\in \left[ \liealg,\liealg
\right]^\perp$. The second term is zero since
\begin{equation*}
\begin{aligned}
g([JX,JY],V) = \eta([JX,JY]) = - d\eta(JX,JY) = -c (\eta \wedge \theta)(JX,JY) =0.
\end{aligned}
\end{equation*}
For the first term we have
\begin{equation*}
\begin{aligned}
(\theta\wedge \Omega)(U,JX,JY) = \Omega(JX,JY) = - g(JX,Y) =0.
\end{aligned}
\end{equation*}
Therefore $g([U,JX],Y)=0$ and $[U,JX] = -JX$ as claimed.

Next, we show $\left[ V,JX \right] =0$. The strategy is the same as before. We
have $[V,JX] \perp X$, since $X \in \left[ \liealg,\liealg \right]^\perp$.
Further
\begin{equation*}
\begin{aligned}
g([V,JX],U) & = \theta([V,JX]) = -d\theta(V,JX) =0\\
g([V,JX],V) & = \eta([V,JX]) = -d\eta(V,JX) = - c(\eta \wedge \theta)(V,JX) =0.
\end{aligned}
\end{equation*}
Now let $Y\in \left\langle U,V,X \right\rangle^\perp$. Then
\begin{equation*}
\begin{aligned}
g([V,JX],Y) & = - \Omega([V,JX], JY)
\\& = d\Omega(V,JX,JY) + \Omega([JX,JY],V) + \Omega([JY,V],JX)
\\ & = (\theta\wedge \Omega)(V,JX,JY) - g([JX,JY],U) - g([JY,V],X).
\end{aligned}
\end{equation*}
The first term is trivially zero.
The last term vanishes since $X\in \left[ \liealg, \liealg \right]^\perp$.
Finally
\begin{equation*}
\begin{aligned}
g([JX,JY],U) &= \theta([JX,JY]) = -d\theta(JX,JY) =0.
\end{aligned}
\end{equation*}
Hence $[V,JX] =0$.

We are ready to show that
\begin{equation}
\label{inequality}
\begin{aligned}
g(\left[ U,\left[ V,X \right] \right] - \left[ V,\left[ U,X \right] \right] - c
\left[ V,X \right], JX) <0
\end{aligned}
\end{equation}
which will produce contradiction, since by Jacobi identity we have
\begin{equation*}
\begin{aligned}
\left[ U,\left[ V,X \right] \right] - \left[ V,\left[ U,X \right] \right] -
c\left[ V,X \right] =
\left[ \left[ U,V \right], X \right] - c\left[ V,X \right] = \left[ cV,X
\right] - c\left[ V,X \right] = 0.
\end{aligned}
\end{equation*}
Let $X,JX, Y_1,\dots, Y_{2n-2}$ be an orthonormal basis of $\hlie = \left\langle U,V
\right\rangle^\perp$. Since $\hlie$ is an ideal of $\liealg$, we have $[V,X]$, $\left[
U,\left[ V,X \right] \right] \in \hlie$.
To compute $g([U,[V,X]],JX)$, we first consider the basis decomposition of $[V,X]$
\begin{equation*}
\begin{aligned}
{} [V,X] & = g([V,X],X)X + g([V,X],JX) JX  + \sum_{k=1}^{2n-2}
g([V,X],Y_k)Y_k.
\end{aligned}
\end{equation*}
Notice that $g([V,X],X)=0$, since $X \in [\liealg,\liealg]^\perp$. Next
we use integrability of $J$ to compute $g([V,X],JX)$. We get
\begin{equation*}
\begin{aligned}
g([V,X],JX)  & = g (-J[JU,X],X)  = g([U,X] - [JU,JX] + J[U,JX], X)
 \\& = g([U,X],X) - g([JU,JX], X) + g(J[U,JX],X).
\end{aligned}
\end{equation*}
The first two terms vanish since $X \in \left[ \liealg, \liealg \right]^\perp$.
Continuing the computation and using $\left[ U,JX \right] = -JX$, we get
\begin{equation}
\label{gvxjx}
\begin{aligned}
g([V,X],JX)  & = g(J[U,JX],X) =  g(J(-JX), X) =  g(X,X) = 1.
\end{aligned}
\end{equation}
Therefore
\begin{equation*}
\begin{aligned}
{} [V,X] & = JX + \sum_{k=1}^{2n-2}
g([V,X],Y_k)Y_k
\end{aligned}
\end{equation*}
Applying $g([U,-],JX)$ to the above equation and using  again that $[U,JX]=-JX$, we obtain
\begin{equation}
\label{guvxjx}
\begin{aligned}
g([U,[V,X]],JX)  = -1 +
\sum_{k=1}^{2n-2}  g([V,X],Y_k)g([U,Y_k],JX). \end{aligned}
\end{equation}
Now we will show $g([U,Z],JX) = - g([V,JX],Z)$ for every $Z\in \left\langle
U,V,X,JX
\right\rangle^\perp$ of length one.
Using integrability of $J$, we get
\begin{equation*}
\begin{aligned}
g([U,Z],JX) & = g(J[JV,Z], X) = g(-[V,Z] + [JV,JZ] - J[V,JZ],X)
\\ & = - g([V,Z],X) + g([JV,JZ],X) + g([V,JZ],JX).
\end{aligned}
\end{equation*}
The first two terms above vanish since $X \in [\liealg,\liealg]^\perp$. Hence
\begin{equation*}
\begin{aligned}
g([U,Z],JX) & = g([V,JZ],JX) = \Omega ([V,JZ],X) \\&  = - d\Omega(V,JZ,X) -
\Omega([JZ,X],V) - \Omega([X,V],JZ)
\\ & = - (\theta \wedge \Omega)(V,JZ,X) + g([JZ,X],U) + g([X,V],Z) \\&  = -
g([V,X],Z),
\end{aligned}
\end{equation*}
where in the last step we used $[JZ,X] \in \liealg^{(1)} \subset U^\perp$.
Using the above formula with $Z=Y_k$, we can write
\eqref{guvxjx} as
\begin{equation}
\label{guvxjx2}
\begin{aligned}
g([U,[V,X]],JX)  = -1 -
\sum_{k=1}^{2n-2}  g([V,X],Y_k)^2.\end{aligned}
\end{equation}
Next we compute $g([V,[U,X]],JX)$. We start by writing $[U,X]$ in the chosen
orthonormal basis of $\hlie$
\begin{equation*}
\begin{aligned}
{} [U,X] & = g([U,X],X)X + g([U,X],JX) JX  + \sum_{k=1}^{2n-2}
g([U,X],Y_k)Y_k.
\end{aligned}
\end{equation*}
Notice that $g([U,X],X) =0$ since $X \in \left[ \liealg,\liealg \right]^\perp$.
As $J$ is integrable, we get
\begin{equation*}
\begin{aligned}
g([U,X],JX) & = g (J[JV,X],X) = g(-[V,X] + [JV,JX] - J[V,JX],X)
\\ & = - g([V,X],X) + g([JV,JX],X)  + g([V,JX],JX) =0,
\end{aligned}
\end{equation*}
 where the first two terms vanish because $X \in [\liealg, \liealg]^\perp$ and the
last term is zero because we showed before that $[V,JX] =0$.
Hence
\begin{equation*}
\begin{aligned}
{} [U,X] & =  \sum_{k=1}^{2n-2}
g([U,X],Y_k)Y_k.
\end{aligned}
\end{equation*}
Applying $g([V,-],JX)$ to the last expression, we obtain
\begin{equation}
\begin{aligned}
{} g([V,[U,X]],JX) & =  \sum_{k=1}^{2n-2}
g([U,X],Y_k)g([V,Y_k],JX).
\end{aligned}
\label{gvuxjx}
\end{equation}
Next, we check that for every $Z \in \left\langle U,V,X,JX
\right\rangle^\perp$,  $g([V,Z],JX)=g([U,X],Z)$ holds.
Applying the integrability of $J$, we have
\begin{equation*}
\begin{aligned}
g([V,Z],JX) & = g(-J[JU,Z],X) = g([U,Z] - [JU,JZ] + J[U,JZ],X)
\\& = g([U,Z],X) - g([JU,JZ],X) - g([U,JZ],JX).
\end{aligned}
\end{equation*}
The first two terms above vanish as $X \in \left[ \liealg, \liealg
\right]^\perp$. Thus we can continue computation as follows
\begin{equation*}
\begin{aligned}
g([V,Z],JX) & =  - g([U,JZ],JX)
= - \Omega([U,JZ],X) \\& = d\Omega(U,JZ,X) + \Omega([JZ,X], U) + \Omega([X,U], JZ)
\\ & = (\theta \wedge \Omega) (U,JZ,X) + g([JZ,X],V) - g([X,U],Z)
\\ & = \Omega(JZ,X) + \eta([JZ,X]) + g([U,X],Z)
\\ & = 0 - d\eta(JZ,X) +  g([U,X],Z) = - c (\eta \wedge \theta)(JZ,X) +  g([U,X],Z)
\\ & = g([U,X],Z).
\end{aligned}
\end{equation*}
With $g([V,Z],JX) = g([U,X],Z)$, the equation \eqref{gvuxjx} becomes
\begin{equation}
\label{gvuxjx2}
\begin{aligned}
{} g([V,[U,X]],JX) & =  \sum_{k=1}^{2n-2}
g([U,X],Y_k)^2.
\end{aligned}
\end{equation}
Now we substitute~\eqref{guvxjx2},\eqref{gvuxjx2}, and \eqref{gvxjx} into the
left side of~\eqref{inequality} and obtain
\begin{equation*}
\begin{aligned}
-1  & - \sum_{k=1}^{2n-2}   g([V,X],Y_k)^2
  - \sum_{k=1}^{2n-2}
g([U,X],Y_k)^2
 - c < -1 + 1 = 0,
\end{aligned}
\end{equation*}
since $c>-1$.
\end{proof}
Besides of proving that every unimodular Kähler Lie algebra is meta-abelian,
Hano also showed in~\cite{hano} that every nilpotent unimodular Kähler Lie
algebra is abelian.
Using this result we can improve conclusion of Proposition~\ref{kaehlerideal}
under the additional assumption that the algebra in question is unimodular.

\begin{theorem}\label{abelianideal}
Let $\liealg$ be an unimodular integrable LCK Lie algebra. Then $\hlie =
\left\langle U,V \right\rangle^\perp$ is a Kähler abelian ideal.
\end{theorem}
\begin{proof}
By Theorems~\ref{unimodularsolvable}--~\ref{uperp}, the nilradical of $\liealg$ is $U^\perp$ and contains
$\hlie$. Hence $\hlie$ is a nilpotent ideal. By Proposition~\ref{kaehlerideal}, $\hlie$ is a Kähler ideal. Hence by
Hano's theorem $\hlie$ is abelian.
\end{proof}

\section{Construction of integrable LCK Lie algebras}
In view of Corollary~\ref{semidirect}, it is natural to ask which semi-direct
products of a Kähler Lie algebra with a two dimensional lie algebra
admit an integrable LCK structure.

For this we fix a constant $c\in \R$ and consider the Lie algebra $\aaac$ with
a basis $U$, $V$ and the Lie bracket $[U,V] = cV$. Notice, that the algebras
with $c\not=0$ are mutually isomorphic and this algebra is usually denoted by
$\aaa$. We equip $\aaac$ with almost Hermitian structure by $g(U,U) =
g(V,V) =1$ and $JU = V$. It is easy to check that this structure is Kähler.

Consider a Kähler Lie algebra $\hlie$  with two derivations $u$, $v$ such that
$[u,v] = cv$. Then we define an action of $\aaac$ on $\hlie$ by $U \mapsto u$,
$V\mapsto v$ and denote the corresponding semi-direct product by
$\aaac\ltimes_{u,v} \hlie$. We endow $\aaac\ltimes_{u,v} \hlie$ with the
product almost Hermitian structure. Notice that since $\aaac\ltimes_{u,v}
\hlie$ is not a direct product of Lie algebras, the resulting fundamental $2$-form on
$\aaac\ltimes_{u,v} \hlie$, in general, is not closed. Also, the integrability
of the almost complex structure on $\aaac\ltimes_{u,v} \hlie$ does not follow
from the integrability on the factors. It is not difficult to verify the following
claim.
\begin{proposition}
Let $\liealg^{2n+2}$ be an integrable LCK Lie algebra, $c = \delta \theta + n$,
$\hlie = \left\langle U,V \right\rangle^\perp$. Denote $\ad_U\!|_\hlie$ and
$\ad_V\!|_\hlie$ by $u$ and $v$, respectively. Then $\liealg$ and
$\aaac\ltimes_{u,v} \hlie$ are isomorphic as almost Hermitian Lie algebras.
\end{proposition}
This shows that every integrable LCK Lie algebra can be constructed as a semi-direct product
$\aaac\ltimes_{u,v} \hlie$. In the next proposition we identify them among all
almost Hermitian Lie algebras of the form $\aaac\ltimes_{u,v} \hlie$.
\begin{proposition}
\label{cond}
Let $\hlie$ be a Kähler Lie algebra, $c\in \R$, and $u$, $v$ derivations on
$\hlie$ such that $[u, v] = cv$.
\begin{enumerate}[1.]
\item The almost complex structure $J$ on $\aaac\ltimes_{u, v}\hlie $ is integrable
if and only if
\begin{equation}
\begin{aligned}
\left[ v + Ju, J \right] =0.
\end{aligned}
\label{vjuj}
\end{equation}
\item The fundamental $2$-form $\Omega$  on $\aaac\ltimes_{u,v}\hlie $
is LCS with the Lee form $\theta=U^\flat$ if and only if
\begin{equation}
\label{jujjuvjjv}
\begin{aligned}
J +u^* J + Ju =0,\quad v^*J + Jv =0,
\end{aligned}
\end{equation}
 where $x^*$
denotes the adjoint operator of $x \in \End_\R(\hlie)$.
\item The almost Hermitian structure on $\aaac\ltimes_{u,v} \hlie$ is
an integrable LCK structure with Lee form $\theta$ if and only if both conditions~\eqref{vjuj} and~\eqref{jujjuvjjv}
hold.
\end{enumerate}
\end{proposition}
\begin{proof}
The almost complex structure $J$ on $\aaac\ltimes_{u,v} \hlie$ is integrable if
and only if $N_J(X,Y) =0$ for every $X \in \aaac$ and $Y \in \hlie$. Indeed, if
both $X$ and $Y$ belong to the same factor then the vanishing of $N_J$ on the pair
$X$, $Y$ follows from the integrability of almost complex structure on this factor.
Next, since for every almost complex structure $N_J(JX,Y) = -JN_J(X,Y)$ and
$N_J$ is bilinear, $J$ is integrable if and only if $N_J(U,Y) =0$ for all
$Y\in \hlie$. Now
\begin{equation*}
\begin{aligned}
N_J(U,Y) & = - [U,Y] + [JU, JY] - J[JU, Y] - J[U,JY]
\\ & = - [U,Y] + [V, JY] - J[V, Y] - J[U,JY]
\\ & =\phantom{-}  uJ^2Y + vJY - JvY - JuJY
\\ & = \phantom{-}  [v + uJ, J]Y.
\end{aligned}
\end{equation*}
Hence $J$ is integrable if and only if $[v + uJ, J] =0$.

Now we compute the $3$-form $d\Omega - \theta \wedge \Omega$  on
$\aaac\ltimes_{u,v} \hlie$. We write $\omega$ for the Kähler $2$-form on
$\hlie$. Since $d\omega=0$, we get that $d\Omega$ evaluated on triples
$X$, $Y$, $Z\in \hlie$ vanishes. It is also clear that $\theta\wedge \Omega$
vanishes on $\bigwedge^3\hlie$.
For every $Z\in \hlie$, we have
\begin{equation*}
\begin{aligned}
(d\Omega - \theta \wedge \Omega)(U,V,Z)&= -\Omega(cV,Z) - \Omega([V,Z],U)
- \Omega([Z,U],V)  - \Omega(V,Z)  =0,
\end{aligned}
\end{equation*}
as $[U,V]=cV$ and $\hlie$ is an ideal orthogonal to $\left\langle U,V \right\rangle$ by construction.
 Finally, let $X$, $Y\in \hlie$. Then
\begin{equation*}
\begin{aligned}
(d\Omega - \theta \wedge \Omega)(U,X,Y) &= - \Omega ([U,X] , Y)  -
\Omega([X,Y],U) - \Omega([Y,U],X) - \Omega(X,Y)
\\ & =  - g(uX, JY) +0+ g(uY, JX) - g(X,JY)
\\ & = g(JuX + u^*JX + JX, Y),\\[2ex]
(d\Omega - \theta \wedge \Omega)(V,X,Y)& = - \Omega ([V,X], Y) - \Omega ([X,Y],V)  - \Omega ([Y,V],X)
\\ &  = - g(vX, JY) + g(vY, JX)
= g(JvX + v^*JX, Y).
\end{aligned}
\end{equation*}
Since $g$ is non-degenerate on $\hlie$, we get that $d\Omega - \theta \wedge
\Omega$ vanishes if and only if $u^* J + Ju + J=0$ and $v^*J + Jv =0$.

The last claim of the proposition is a consequence of the first two.
\end{proof}
Denote by $H_{n,c}$ the class of triples $(\hlie, u, v)$ where $\hlie$ is a
Kähler
Lie algebra of dimension $2n$, and $u$, $v$ are derivations on $\hlie$ such that
\begin{equation}
\label{conditions}
\begin{aligned}
{}[u,v] = cv,\ [v + uJ, J]=0,\ v^*J + Jv =0,\ J + u^*J + Ju =0.
\end{aligned}
\end{equation}
We say that $(\hlie_1, u_1,v_1)$ and $(\hlie_2,u_2,v_2) \in H_{n,c}$ are
isomorphic if there is an isomorphism $\phi \colon \hlie_1 \to \hlie_2$ of
Kähler Lie algebras, such that $u_2 \circ \phi = \phi \circ u_1$ and $v_2 \circ
\phi = \phi \circ v_1$.
It is routine to check the following result.
\begin{proposition}
Let $(\hlie_1, u_1,v_1)$ and $(\hlie_2,u_2,v_2) \in H_{n,c}$.
Then the LCK Lie algebras $\aaac\ltimes_{u_1,v_1} \hlie_1$ and
$\aaac\ltimes_{u_2,v_2} \hlie_2$ are isomorphic if and only if $(\hlie_1,u_1,v_1)$
and $(\hlie_2,u_2,v_2)$ are isomorphic.
\end{proposition}
Denote by $A_{n,c}$ the subclass of $H_{n,c}$ consisting of triples $(\hlie,
u,v)$ with abelian $\hlie$.
Next we describe unimodular integrable LCK Lie algebras in terms of triples in
$H_{n,c}$.
\begin{theorem}\label{charc}
Let $(\hlie,u,v)\in A_{n,n}$. Then the Lie algebra $\aaan\ltimes_{u, v} \hlie$ is  unimodular. Every unimodular integrable LCK Lie algebra is isomorphic to
$\aaan\ltimes_{u, v} \hlie$ for some $(\hlie, u, v) \in A_{n,n}$.
\end{theorem}
\begin{proof}
The second claim of the theorem is a direct consequence of Corollary~\ref{semidirect},
Theorem~\ref{abelianideal}, and Proposition~\ref{cond}.

 Now, we have to show that $\liealg = \mathfrak{r}_{2,n}\ltimes_{u, v} \hlie$ is
unimodular if $(\hlie, u, v) \in A_{n,n}$.
Since $\hlie$ is an abelian ideal we get $\trace(\ad_x) = \trace(\ad_x\!|_\hlie)=0$ for every $x \in \hlie$.
Next $\trace(\ad_U) = n + \trace(u)$ and $\trace(\ad_V) = \trace(v) $, since
$[U,V]=nV$ and $u = \ad_U\!|_\hlie$, $v = \ad_V\!|_\hlie$.
We have $\trace(nv) = \trace([u, v]) =0$, hence also $\trace(v) =0$.
Now, multiply $J + u^*J + Ju=0$ with $(-J)$ from the left. We get $\id - Ju^*J + u=0$.
Thus
\begin{equation*}
\begin{aligned}
\trace(u) = -2n + \trace(Ju^*J) = -2n + \trace(J^2u^*) = -2n - \trace(u^*) = -2n
- \trace(u).
\end{aligned}
\end{equation*}
Hence $\trace(u) = -n$ and $\trace(\ad_U\!|_\hlie) =n -n =0 $. This shows that
$\mathfrak{r}_{2,n} \ltimes_{u, v} \hlie$ is unimodular.
\end{proof}
\begin{remark}
There is a one-to-one correspondence between isoclasses in $A_{n,c}$ and
$A_{n,1}$ for any $c\not=0$ given by
\begin{equation}
\label{correspondence}
\begin{aligned}
A_{n,c} & \to A_{n,1} \\
(\hlie, u,v) & \mapsto \left(\hlie, \frac{1}{c} u + \frac{1-c}{2c}\, \id,
\frac{1}{c} v\right).
\end{aligned}
\end{equation}
Now, we can define direct sums on $A := \coprod_{n\ge 0} A_{n,1}$ and define the
notion of indecomposable elements. Collecting all the previous results we see
that the classification of unimodular integrable LCK Lie algebras up to LCK
Lie algebra isomorphism is reduced to the classification of indecomposable
triples in $A$. We will treat this problem in a separate article.

Notice that we cannot extend the correspondence~\eqref{correspondence} to a
correspondence between $H_{n, c}$ and $H_{n,1}$. Indeed, if
$u$ is a derivation on $\hlie$, then
$\frac{1}{c} u + \frac{1-c}{2c} \id$ is a derivation on $\hlie$ if and only if $\id$ is a
derivation on $\hlie$. But this is possible only if $\hlie$ is abelian.
\end{remark}
The next example shows that $A_{n, c}$ can be  a proper subset of $H_{n, c}$.
We take $n=1$, $c=-1$, and $\hlie$ to be the Kähler algebra with the basis
$A$, $B$ such that
\begin{equation*}
\begin{aligned}
g(A,A) = g(B,B) =1,\ JA = B,\ [A,B] = A.
\end{aligned}
\end{equation*}
We define the operators $u$ and $v$ on $\hlie$ by
\begin{equation*}
\begin{aligned}
uA =-A,\ uB =0,\ vA = 0,\ vB =-A.
\end{aligned}
\end{equation*}
The operators $u$ and $v$
are derivations on $\hlie$
\begin{equation*}
\begin{aligned}
u[A,B] = uA = -A = [-A,B] = [uA,B] + [A, uB],\\
v[A,B] = vA = 0 = [vA,B] + [A,-A] = [vA,B] + [A,vB].
\end{aligned}
\end{equation*}
Next, we check that $u$ and $v$ satisfy the conditions~\eqref{conditions}.
We have
\begin{equation*}
\begin{aligned}
([u, v] + v) A&  = -vuA = v A =0,\\  ([u, v]+v)B & = uvB + vB = u(-A) -A = A -A =0.
\end{aligned}
\end{equation*}
Hence $[u, v]=-v$. Moreover,
\begin{equation*}
\begin{aligned}
(J + u^*J + Ju) A & = B + u^*B - JA = B + 0 - B =0,\\
(J + u^*J + Ju) B & = -A - u^*A +0 = -A + A =0,\\
(v^*J + Jv) A & = v^* B + 0 = 0, \\
(v^* J + Jv) B & = -v^* A - JA = B - B =0.
\end{aligned}
\end{equation*}
It is left to show that $v + uJ$ commutes with $J$. But, in fact, $v + uJ=0$
\begin{equation*}
\begin{aligned}
(v + uJ) A & = 0 + uB  = 0, \\
(v + uJ) B & = -A -uA =  - A + A =0.
\end{aligned}
\end{equation*}

\section{Four-dimensional unimodular integrable LCK Lie algebras.}
The description of all four-dimensional unimodular integrable LCK Lie algebras
can be done either by filtering the list of all $4$-dimensional LCK Lie
algebras obtained in~\cite{angella} or by classifying isomorphism classes in
$A_{1,1}$.

We go by the second path as it involves less computation.
Notice that for every $(\hlie,u,v) \in A_{1,1}$ the operator $v$ is nilpotent.
Indeed, $\ad_V$ is a nilpotent operator on $\liealg = \aaan\ltimes_{u,v}\hlie $, since
$V \in U^\perp=\liealg^{(1)} $ and $\liealg^{(1)}$ is a nilpotent ideal, as $\liealg$ is solvable due to Theorem~\ref{unimodularsolvable}.
Hence also the restriction of $\ad_V$ on $\hlie$ is nilpotent.

We distinguish two cases. The first one when $v=0$ and the second when
$\rank(v)=1$.

For $v=0$ the conditions~\eqref{conditions} reduce to
\begin{equation*}
\begin{aligned}
{} [Ju,J]=0,\quad J + u^*J + Ju =0.
\end{aligned}
\end{equation*}
The first condition is equivalent to $[u,J]=0$.
Let $X$, $JX$ be an orthonormal basis of $\hlie$. Then by an easy computation from $[u,J]=0$ and $J + u^*J +
Ju =0$  we obtain that
the matrix of $u$ in this basis has the form
\begin{equation*}
\begin{aligned}
\left(
\begin{array}{rr}
-\frac{1}{2} & -b \\[1ex]
b & -\frac{1}{2}
\end{array}
 \right).
\end{aligned}
\end{equation*}
Then, for every $Y\in \hlie$ of length one we get
$g(uY,JY) = b$ and $g(uY, Y) = -(1/2)$.
Hence to each triple $(\hlie, u, v)$ we can associate a real number $b\in \R$.
It can be shown that two triples are isomorphic if and only if they have the
same parameter $b$. So we get a one-parameter family of integrable LCK Lie
algebras.
Denote by $\liealg_b$ the LCK Lie algebra constructed from a triple
$(\hlie, u, v)$ with the parameter $b$.
The family of Lie algebras $\liealg_b$ can be identified with the family
$\mathfrak{r}'_{4,\gamma,\delta}$ in~\cite{angella}, where the extra parameter
$\delta>0$ is due to the lack of normalization of $U$.
 The  family $\liealg_b$  was thoroughly  studied in Section~3.3.2
of~\cite{andrada}.
In particular, it was identified for which parameters $b$ the $1$-connected Lie
group associated to $\liealg_b$ admits a cocompact discrete subgroup.

In the case $\rank(v)=1$ the subspace $\ker v$ is one-dimensional. Choose
$X \in \ker v$ of length one. Then the relations~\eqref{conditions} imply
that the matrices of $u$ and $v$ in the basis $X$, $JX$ are
\begin{equation}
\label{uvmatr}
\begin{aligned}
{ } [u] =
\left(
\begin{array}{rr}
0 & 0 \\
0 & -1
\end{array}
 \right)\quad \quad \quad
[v] =
\left(
\begin{array}{rr}
0 & 1 \\
0 & 0
\end{array}
 \right).
\end{aligned}
\end{equation}
This shows that there is exactly one isomorphism class of such triples
$(\hlie, u,v)$ with $\rank(v)=1$.
The corresponding Lie algebra is denoted by $\mathfrak{d}_4$ in the list of LCK
Lie algebras of~\cite{angella}. The corresponding $1$-connected Lie group
contains a cocompact discrete subgroup. This fact was first claimed
in~\cite{cordero}, with erroneous justification, and then shown to be true
in~\cite{sawai}.

\section{Final remarks}
It is possible to produce many examples of unimodular integrable LCK Lie
algebras in any dimension $2n+2$. For example, we can start with the triple
$(\hlie, u, v)\in A_{1,1}$, where $u$ and $v$ are given by~\eqref{uvmatr}.
Then $(\hlie^{\oplus n}, u^{\oplus n} v^{\oplus n}) \in A_{n,1}$.
Using the correspondence~\eqref{correspondence} we get a triple
$(\hlie^{\oplus n}, \bar{u}^{\oplus n}, \bar{v}^{\oplus n})\in A_{n,n}$ where the matrices
of $\bar{u}$ and $\bar{v}$ are

\begin{equation*}
\begin{aligned}
{} [\bar{u}] =
\left(
\begin{array}{cc}
\frac{1}{2} (n-1) & 0 \\[1ex]
0 & - \frac{1}{2}(n+1)
\end{array}
 \right)\quad \quad \quad
[\bar{v}] =
\left(
\begin{array}{rr}
 0 & n  \\[1ex]
0 & 0
\end{array}
 \right).
\end{aligned}
\end{equation*}
Given a solvable Lie group $G$ with right-invariant LCK structure, it is
always possible to construct a quotient manifold $M$ of $G$ such that the
resulting LCK structure on $M$ is not globally conformal Kähler.
For this we take the discrete subgroup $\Gamma = \exp(\Z U)$ of $G$ and define
$M = G/\Gamma$. We consider the induced LCK structure on $M$. Define the closed
loop $\gamma \colon [0,1] \to M$ by $\gamma(t) = \pi \circ\exp(Ut)$, where $\pi
\colon G \to M$ is the canonical projection.
Then the integral of the Lee 1-form along $\gamma$ is given by
\begin{equation*}
\begin{aligned}
\int_{\gamma} \theta = \int_{0}^1 \theta (\dot\gamma(t)) dt = \int_{0}^1
\theta(U_{\gamma(t)}) dt = 1.
\end{aligned}
\end{equation*}
If $\theta$ on $M$ were exact, say $\theta = d\alpha $ then we would have
by Stokes theorem
\begin{equation*}
\begin{aligned}
\int_{\gamma} \theta = \int_{\gamma} d\alpha = \int_{\partial \gamma} \alpha =0.
\end{aligned}
\end{equation*}
Hence $\theta$ is not exact on $M$ and the LCK structure on $M$ is not globally
conformal Kähler.

Combining the above two constructions, we get examples of integrable LCK manifolds, which are not
globally conformal Kähler in every even dimension greater  than $2$.

\bibliography{lck}

\providecommand{\bysame}{\leavevmode\hbox to3em{\hrulefill}\thinspace}
\providecommand{\MR}{\relax\ifhmode\unskip\space\fi MR }
\providecommand{\MRhref}[2]{%
  \href{http://www.ams.org/mathscinet-getitem?mr=#1}{#2}
}
\providecommand{\href}[2]{#2}
\begin{thebibliography}{dACFM89}

\bibitem[AO18]{andrada}
A.~Andrada and M.~Origlia, \emph{Lattices in almost abelian {L}ie groups with
  locally conformal {K}\"{a}hler or symplectic structures}, Manuscripta Math.
  \textbf{155} (2018), no.~3-4, 389--417. \MR{3763412}

\bibitem[AO20]{angella}
Daniele Angella and Marcos Origlia, \emph{Locally conformally {K}\"{a}hler
  structures on four-dimensional solvable {L}ie algebras}, Complex Manifolds
  \textbf{7} (2020), no.~1, 1--35. \MR{4034635}

\bibitem[Bel00]{belg00}
Florin~Alexandru Belgun, \emph{On the metric structure of non-{K}\"{a}hler
  complex surfaces}, Math. Ann. \textbf{317} (2000), no.~1, 1--40. \MR{1760667}

\bibitem[dACFM89]{cordero}
Luis~C. de~Andr\'{e}s, Luis~A. Cordero, Marisa Fern\'{a}ndez, and Jos\'{e}~J.
  Menc\'{\i}a, \emph{Examples of four-dimensional compact locally conformal
  {K}\"{a}hler solvmanifolds}, Geom. Dedicata \textbf{29} (1989), no.~2,
  227--232. \MR{988272}

\bibitem[DO98]{DrOr}
Sorin Dragomir and Liviu Ornea, \emph{Locally conformal {K}\"ahler geometry},
  Progress in Mathematics, vol. 155, Birkh\"auser Boston, Inc., Boston, MA,
  1998. \MR{1481969}

\bibitem[Han57]{hano}
Jun-ichi Hano, \emph{On {K}aehlerian homogeneous spaces of unimodular {L}ie
  groups}, Amer. J. Math. \textbf{79} (1957), 885--900. \MR{95979}

\bibitem[Saw07]{sawai}
Hiroshi Sawai, \emph{A construction of lattices on certain solvable {L}ie
  groups}, Topology Appl. \textbf{154} (2007), no.~18, 3125--3134. \MR{2364640}

\bibitem[Tri82]{tri82}
Franco Tricerri, \emph{Some examples of locally conformal {K}\"{a}hler
  manifolds}, Rend. Sem. Mat. Univ. Politec. Torino \textbf{40} (1982), no.~1,
  81--92. \MR{706055}

\bibitem[Vai79]{vaisman_roma}
Izu Vaisman, \emph{Locally conformal {K}\"ahler manifolds with parallel {L}ee
  form}, Rend. Mat. (6) \textbf{12} (1979), no.~2, 263--284. \MR{557668}

\bibitem[Vai85]{lcs-Vai}
\bysame, \emph{Locally conformal symplectic manifolds}, Internat. J. Math.
  Math. Sci. \textbf{8} (1985), no.~3, 521--536. \MR{809073}

\end{thebibliography}
\bibliographystyle{amsalpha}
\end{document}